\documentclass[reqno]{amsart}
\usepackage{graphicx}
\usepackage{geometry}
\usepackage{amsmath}
\usepackage{amscd}
\usepackage{amsthm}
\usepackage{amsfonts}
\usepackage{amssymb}
\usepackage[colorlinks=true]{hyperref}
\newtheorem{theorem}{Theorem}[section]
\newtheorem{corollary}[theorem]{Corollary}
\newtheorem{proposition}[theorem]{Proposition}

\theoremstyle{definition}
\newtheorem{definition}[theorem]{Definition}
\newtheorem{example}[theorem]{Example}

\DeclareMathOperator{\Hol}{Hol}
\DeclareMathOperator{\End}{End}
\DeclareMathOperator{\Aut}{Aut}

\DeclareMathOperator{\Perm}{Perm}

\newcommand{\B}{\mathfrak{B}}

\numberwithin{equation}{section}

\begin{document}

\title[Braces and Hopf-Galois structures]{Skew left braces and isomorphism problems for Hopf-Galois structures on Galois extensions}
\author{Alan Koch and Paul J.~Truman}
\date{\today        }

\begin{abstract}
Given a finite group $ G $, we study certain regular subgroups of the group of permutations of $ G $, which occur in the classification theories of two types of algebraic objects: skew left braces with multiplicative group isomorphic to $ G $ and Hopf-Galois structures admitted by a Galois extension of fields with Galois group isomorphic to $ G $. We study the questions of when two such subgroups yield isomorphic skew left braces or Hopf-Galois structures involving isomorphic Hopf algebras. In particular, we show that in some cases the isomorphism class of the Hopf algebra giving a Hopf-Galois structure is determined by the corresponding skew left brace. We investigate these questions in the context of a variety of existing constructions in the literature. As an application of our results we classify the isomorphically distinct Hopf algebras that give Hopf-Galois structures on a Galois extension of degree $ pq $ for $ p>q $ prime numbers. 
\end{abstract}

\maketitle

\section{Introduction} \label{section_introduction}

Let $ G $ be a finite group and let $ \Perm(G) $ denote the group of permutations of $ G $. A subgroup $ N \leq \Perm(G) $ is said to be {\em regular} if $ |N|=|G| $, the action of $ N $ on $ G $ is transitive, and the stabilizer in $ N $ of every $ \sigma \in G $ is trivial (any two of these conditions guarantees the third). One example of a regular subgroup of $ \Perm(G) $ is the image of $ G $ under the left regular representation $ \lambda : G \hookrightarrow \Perm(G) $. This map also yields an action of $ G $ on $ \Perm(G) $ by $ \,^{\sigma}\pi = \lambda(\sigma)\pi\lambda(\sigma)^{-1} $, and this paper is concerned with regular subgroups $ N \leq \Perm(G) $ that are {\em stable} under this action. These subgroups are of interest because they occur in the classification theories of two types algebraic objects. 

On one hand, there is a correspondence (although not a bijection) between $ G $-stable regular subgroups of $ \Perm(G) $ and {\em skew left braces} with multiplicative group isomorphic to $ G $. Each of these yields a set theoretic solution to the Yang-Baxter equation on the underlying set $ G $, each of which extends naturally to a solution on the vector space $ K[G] $ (for a given field $ K $). Two regular $ G $-stable subgroups can correspond to isomorphic skew left braces, and so we obtain a partition of the set of regular $ G $-stable subgroups of $ \Perm(G) $. In terms of the Yang-Baxter equation, the solutions arising from isomorphic braces are equivalent up to a change of basis of $ K[G] $. 
 
On the other hand, by a theorem of Greither and Pareigis $ G $-stable regular subgroups of $ \Perm(G) $ correspond bijectively with {\em Hopf-Galois structures} admitted by a Galois extension of fields $ L/K $ with Galois group isomorphic to $ G $, each consisting of a $ K $-Hopf algebra and a certain $ K $-linear action of $ H $ on $ L $. Applications of Hopf-Galois structures include  the formulation of variants of the Galois correspondence and the study of integral module structure in extensions of local or global fields. Two distinct Hopf-Galois structures can involve isomorphic Hopf algebras (alternatively, we might view this as two distinct actions of a single Hopf algebra on $ L $); it is possible to detect when this occurs in purely group theoretic terms, and so we obtain another partition of the set of regular $ G $-stable subgroups of $ \Perm(G) $.  

In this paper we address the natural question of comparing the two notions of isomorphism discussed above via the corresponding partitions of the set of regular $ G $-stable subgroups of $ \Perm(G) $. In Section \ref{sec_HGS_SLB} we discuss the connection between skew left braces and Hopf-Galois structures in more detail, and in Section \ref{section_partinioning} we recall and reformulate existing criteria for two regular $ G $-stable subgroup of $ \Perm(G) $ to correspond to isomorphic skew left braces or to Hopf-Galois structures involving isomorphic Hopf algebras. We find that neither of these notions implies the other in general, but in Sections \ref{sec_rho_conjugate} - \ref{sec_opposites} we show that they have rich interactions with various existing constructions, including regular $ G $-stable subgroups arising from {\em abelian maps}, as studied by Childs in \cite{Ch13} and generalized by the first named author in \cite{Koc21}, and the notions of {\em opposite} Hopf-Galois structures and skew left braces, as studied by the authors in \cite{KT20}. Finally, in Section \ref{sec_brace_classifications} we undertake a detailed study of the case in which $ |G|=pq $ with $ p>q $ prime numbers. The classifications of $ G $-stable regular subgroups, skew left braces, and Hopf-Galois structures are known in this case; by applying our techniques we identify the isomorphically distinct Hopf algebras that occur.

\section{Hopf-Galois structures and Skew left braces} \label{sec_HGS_SLB}

In this section we describe the connections between Hopf-Galois structures, $ G $-stable regular subgroups of $ \Perm(G) $, and skew left braces. For more detailed summaries we refer to the reader to \cite[Section 2]{KT20} and \cite[Appendix A]{SV18}

\subsection{Hopf-Galois structures}\label{subsection_HGS}

A {\em Hopf-Galois structure} on a finite extension of fields $ L/K $ consists of a cocommutative $ K $-Hopf algebra $ H $ and an action of $ H $ on $ L $ making $ L $ into an $ H $-module algebra and such that the $K$-linear map $j:L\otimes H\to\End_K(L)$ given by $j(x\otimes h)(y)=x(h \cdot y)$ for all $ h \in H $ and all $ x,y \in L $ is bijective (see \cite[Definition 2.7]{Ch00}).

Greither and Pareigis \cite{GP87} classify the Hopf-Galois structures admitted by a finite separable extension of fields in group theoretic terms. Specializing to the case in which $ L/K $ is a Galois extension with Galois group $ G $, their theorem states that there is a bijection between Hopf-Galois structures admitted by $ L/K $ and regular $ G $-stable subgroups of $ \Perm(G) $. Given such a subgroup $ N $, the Hopf algebra giving the Hopf-Galois structure corresponding to $ N $ is $ H_{N}:=L[N]^{G} $, where $ G $ acts on $ L $ as Galois automorphisms and on $ N $ via $ \,^{\sigma}\eta = \lambda(\sigma)\eta\lambda(\sigma)^{-1} $ (the assumption that $ N $ is $ G $-stable ensures that this is indeed an action of $ G $ on $ N $). The theorem of Greither and Pareigis also specifies the action of $ H_{N} $ on $ L $, but we shall not need this information in what follows. 

\begin{example} \label{example_classical_HGS}
The image of the right regular representation $ \rho : G \rightarrow \Perm(G) $ is a regular subgroup of $\Perm(G)$, and is $ G $-stable since we have $ \,^{\sigma}\rho(\tau) =\rho(\tau)$ for all $ \sigma,\tau \in G $. The corresponding Hopf-Galois structure is given by the Hopf algebra $ K[G] $, along with its natural action on $ L $. We call this the {\em classical} Hopf-Galois structure on $ L/K $. 
\end{example}

\begin{example}\label{example_canonical_nonclassical_HGS}
The image of the left regular representation $ \lambda : G \rightarrow \Perm(G) $ is a regular subgroup of $\Perm(G)$, and is $ G $-stable since we have $ \,^{\sigma}\lambda(\tau)=\lambda(\sigma \tau \sigma^{-1})$ for all $ \sigma,\tau \in G $. If $ G $ is nonabelian then $ \lambda(G) \neq \rho(G) $, and so $ \lambda(G) $ corresponds to a different Hopf-Galois structure on $ L/K $, with Hopf algebra $ L[\lambda(G)]^{G} $. We call this the {\em canonical nonclassical} Hopf-Galois structure on $ L/K $.
\end{example}

In each of these examples the regular subgroup $ N $ is isomorphic to $ G $. However, this need not be the case:

\begin{example}\label{example_D_3_HGS}
Let $ L/K $ be a Galois extension with Galois group 
\[ G = \langle \sigma, \tau \mid \sigma^{3}=\tau^{2}=1, \; \tau \sigma \tau^{-1} = \sigma^{-1} \rangle \cong D_{3}. \]
For $ c=0,1,2 $ let $ N_{c} = \langle \lambda(\sigma),\rho(\sigma^{c}\tau) \rangle $. We find (\cite[Lemma 1]{KKTU19b} or a routine verification) that each $ N_{c} $ is a distinct cyclic regular subgroup of $ \Perm(G) $, and is also $ G $-stable: both generators of $ G $ act trivially on $ \rho(\sigma^{c}\tau) $, and we have $ \,^{\sigma}\lambda(\sigma) = \lambda(\sigma) $ and $ \,^{\tau}\lambda(\sigma)=\lambda(\sigma^{-1}) $. Thus the dihedral extension $ L/K $ admits three Hopf-Galois structures for which the corresponding regular $ G $-stable subgroup of $ \Perm(G) $ is cyclic.   
\end{example}

In Example \ref{example_D_3_HGS} we have described the cyclic groups $ N_{c} $ using two generators, following \cite{By04} and \cite{AB20}; we will continue to adopt this slightly unconventional notation in order to relate our results to the results of those papers. 

If $ N $ is a regular $ G $-stable subgroup of $ \Perm(G) $ then we refer to the isomorphism class of $ N $ as the {\em type} of the corresponding Hopf-Galois structure. For example: the previous example provides us with Hopf-Galois structures of cyclic type on a dihedral extension of degree $ 6 $. 

In \cite{By96} Byott shows that the question of determining all regular $ G $-stable subgroups of $ \Perm(G) $ that are isomorphic to a given group $ N $ is closely related to the question of determining all regular subgroups of the holomorph $\Hol(N)$ of $N$, where $ \Hol(N) = N \rtimes \Aut(N) $, that are isomorphic to $ G $; the latter is often an easier problem since $ \Hol(N) $ is a smaller group that $ \Perm(G) $. Various authors have enumerated and described the Hopf-Galois structure admitted by a Galois extension with prescribed Galois group $ G $; see for example \cite{Ko98}, \cite{By04}, \cite{Ze19}. Others have developed more general methods for creating or describing families of regular $ G $-stable subgroups of $ \Perm(G) $: see for example \cite{Ch13}, \cite{CRV16}, \cite{KT20}. We will describe some of these constructions in more detail in subsequent sections. 

\subsection{Skew left braces} \label{subsection_skew_left_braces}

A {\em skew left brace} is a triple $ \B = (B,\cdot,\circ) $ such that $ (B,\cdot) $ and $ (B,\circ) $ are finite groups whose operations satisfy the {\em brace relation}
\begin{equation} \label{eqn_brace_relation}
x \circ(y \cdot z) = (x \circ y) \cdot x^{-1} \cdot (x \circ z) \mbox{ for all } x,y,z \in B, 
\end{equation}
where $ x^{-1} $ denotes the inverse of $ x $ in the group $ (B,\cdot) $. 

A consequence of the brace relation \eqref{eqn_brace_relation} is that $ (B,\cdot) $ and $ (B,\circ) $ share the same identity element, but in general the inverse of an element $ x $ in the group $ (B,\circ) $ (denoted $ \overline{x} $) is not equal to $ x^{-1} $. 

We call the groups $ (B,\cdot) $ the {\em dot group}, and $ (B,\circ) $ the {\em circle group} of the skew left brace $ \B $ (elsewhere in the literature these are sometimes called the {\em additive group} and {\em multiplicative group} of $ \B $). To ease notation, we write $ x \cdot y = xy $ where there is no danger of confusion. For brevity, we shall henceforth refer to a skew left brace simply as a {\em brace}, but we note that in the historical development of the subject this term originally applied to skew left braces with abelian dot group. 

\begin{example} \label{example_brace_trivial}
Let $ (B,\cdot) $ be a finite group and let $ x \circ y =  x \cdot y $ for all $ x,y \in B $. Then $ \B = (B,\cdot,\circ) $ is a brace, called the {\em trivial} brace for $ (B,\cdot) $.
\end{example}

\begin{example} \label{example_brace_almost_trivial}
Let $ (B,\cdot) $ be a finite group and let $ x \circ y =  y \cdot x $ for all $ x,y \in B $. Then $ \B = (B,\cdot,\circ) $ is a brace, called the {\em almost trivial} brace for $ (B,\cdot) $.
\end{example}

In each of these examples the dot and circle groups of $ \B $ are isomorphic to each other. However, this need not be the case:

\begin{example} \label{example_brace_C6_D3}
Let $ (B,\cdot) $ be a cyclic group of order $ 6 $, presented using two generators: 
\[ (B,\cdot) = \langle x, y \mid x^{3}=y^{2}=1, \; y x y^{-1}= x \rangle, \]
and let 
\[ x^{i}y^{j} \circ x^{k}y^{\ell} = x^{i+(-1)^{j}k} y^{k + \ell}. \]
It is routine to verify that $ (B, \circ) $ is a group in which $ x $ has order $ 3 $, $ y $ has order $ 2 $, and $ y \circ x \circ \overline{y} = \overline{x} $, whence $ (B,\circ) \cong D_{3} $. Moreover, $ \B=(B,\cdot,\circ) $ is a brace. 
\end{example}

In \cite[Proposition 1.11]{GV17}, Guanrnieri and Vendramin show that, given groups $ N, G $ of the same order, braces $ \B=(B,\cdot,\circ) $ with $ (B,\cdot) \cong N $ and $ (B,\circ) \cong G $ correspond with {\em bijective $ 1 $-cocycles} $ G \rightarrow N $. In \cite[Appendix A]{SV18} this correspondence is reformulated in terms of regular subgroups of $ \Hol(N) $ that are isomorphic to $ G $. Since our focus is on $ G $-stable regular subgroups of $\Perm(G)$ rather than regular subgroups of $\Hol(N)$, we reformulate the correspondence in this framework, as follows:  

Firstly, let $ \B=(B,\cdot,\circ) $ be a brace with $ (B,\circ) \cong G $. For each $ x \in B $, the map $ \eta_{x} : B \rightarrow B $ defined by $ \eta_{x}(y) = x \cdot y $ is a permutation of $ B $, and since $ (B,\cdot) $ is a group the set $ N_{\B}=\{ \eta_{x} \mid x \in B \} $ is a regular subgroup of $ \Perm(B) $. By using the brace relation \eqref{eqn_brace_relation}, it can be shown that $ N_{\B} $ is a $ (B,\circ) $-stable subgroup of $ \Perm(B) $. Identifying $ (B,\circ) $ with $ G $, we obtain a regular $ G $-stable subgroup of $ \Perm(G) $.  

Conversely, let $ N $ be a regular $ G $-stable subgroup of $ \Perm(G) $.  The regularity of $ N $ implies that the map $ a : N \rightarrow G $ defined by $ a(\eta) = \eta[1_{G}] $ for all $ \eta \in N $ is a bijection. Define a new binary operation on $ N $ by
\[ \eta \circ \pi = a^{-1}(a(\eta)a(\pi)) \mbox{ for all } \eta, \pi \in N, \]
where the multiplication inside the brackets takes place in $ G $. Then $ (N,\circ) $ is a group isomorphic to $ G $ and, since $ N $ is a $ G $-stable subgroup of $ \Perm(G) $, the brace relation \eqref{eqn_brace_relation} is satisfied. Therefore $ \B_{N} = (N,\cdot,\circ) $ is a brace with $ (N,\circ) \cong G $. In \cite{CS21} this construction is referred to as {\em transport of structure}. 

Alternatively, we may define a new binary operation on $ G $ by
\[ \sigma \cdot \tau = a(a^{-1}(\sigma)a^{-1}(\tau)) \mbox{ for all } \sigma, \tau \in G, \]
where the multiplication inside the brackets takes place in $ N $. Then $ (G,\cdot) $ is a group isomorphic to $ N $, and $ (G,\cdot,\circ) $ is a brace, isomorphic to the brace $ \B_{N} $  constructed above via $ a^{-1} $. 

In Section \ref{section_partinioning} we will discuss precise criteria for two $ G $-stable subgroups of $ \Perm(G) $ to yield isomorphic braces. 

\begin{example}
We have seen in Example \ref{example_classical_HGS} that the image of the right regular representation $ \rho : G \rightarrow \Perm(G) $ is a regular $ G $-stable subgroup of $\Perm(G)$. The corresponding bijection $a:\rho(G)\to G$ is given by $a(\rho(\sigma))=\sigma^{-1}$, and the resulting circle operation is given by
\[\rho(\sigma)\circ \rho(\tau) = \left(\rho(\sigma)^{-1}\rho(\tau)^{-1}\right)^{-1} = \rho(\tau)\rho(\sigma).\]
Thus the subgroup $ \rho(G) $ corresponds to the almost trivial brace for $ G $ (see Example \ref{example_brace_almost_trivial}).
\end{example}

\begin{example}
We have seen in Example \ref{example_canonical_nonclassical_HGS} that the image of the left regular representation $ \rho : G \rightarrow \Perm(G) $ is a regular $ G $-stable subgroup of $\Perm(G)$. The corresponding bijection $a:\lambda(G)\to G$ is given by $a(\lambda(\sigma)) = \sigma $, and the resulting circle operation is simply $\lambda(\sigma)\circ \lambda(\tau)=\lambda(\sigma \tau)$. Therefore the corresponding brace is the trivial brace for $ G $ (see Example \ref{example_brace_trivial}). 
\end{example}

\begin{example} \label{example_D_3_brace_equivalence}
Let $ G \cong D_{3} $ as in Example \ref{example_D_3_HGS}, and consider the regular $ G $-stable subgroup $ N_{0} \leq \Perm(G) $ constructed in that example. Let $ \eta = \lambda(\sigma) $ and $ \pi = \rho(\tau) $. The corresponding bijection $ a : N_{0} \to G $ is given by $ a(\eta^{i} \pi^{j}) = \sigma^{i}\tau^{-j} = \sigma^{i}\tau^{j} $ (since $ \tau $ has order $ 2 $). The resulting circle operation is given by
\begin{eqnarray*}
\eta^{i} \pi^{j} \circ \eta^{k} \pi^{\ell} & = & a^{-1}(\sigma^{i}\tau^{j} \sigma^{k}\tau^{\ell}) \\
& = & a^{-1}(\sigma^{i} \sigma^{k(-1)^{j}} \tau^{j} \tau^{\ell}) \\
& = & \eta^{i+k(-1)^{j}} \pi^{j+\ell}.
\end{eqnarray*}
Therefore the corresponding brace is the brace constructed in Example \ref{example_brace_C6_D3}. By similar calculations it can be shown that the subgroups $ N_{1}, N_{2} $ of Example \ref{example_D_3_HGS} also correspond to this brace. We shall see a more illuminating explanation of this fact in Section \ref{sec_rho_conjugate}.
\end{example}

\section{Brace equivalence and Hopf algebra isomorphisms} \label{section_partinioning}

In our discussion of the relationship between $ G $-stable regular subgroups $ N $ of $ \Perm(G) $ and braces $ \B = (B,\cdot, \circ) $ with $ (B,\cdot) \cong N $ and $ (B,\circ) \cong G $ (Subsection \ref{subsection_skew_left_braces}) we noted that multiple such subgroups can correspond to the same brace. This is made precise in \cite[Proposition 4.3]{GV17} and \cite[Appendix A]{SV18} in terms of regular subgroups of $ \Hol(N) $ that are isomorphic to $ G $. We prefer to formulate this concept in terms of regular $ G $-stable subgroups of $ \Perm(G) $, since we feel this framework is more suitable for comparing it with the notion of Hopf algebra isomorphism mentioned in Section \ref{section_introduction}, and discussed in detail later in this section. Our approach is similar to \cite[Proposition 2.1]{Ze19}, but we give a self-contained proof for the convenience of the reader. We fix an identification of $ G $ with $ (B,\circ) $; this identifies $ \Perm(G) $ with $ \Perm(B) $, $ \Aut(G) $ with $ \Aut(B,\circ) $, and regular $ G $-stable subgroups of $ \Perm(G) $ with regular $ (B,\circ) $-stable subgroups of $ \Perm(B) $.

\begin{proposition} \label{prop_parameterization}
Let $ \B = (B,\cdot,\circ) $ be a brace and let $ N_{\B} = \{ \eta_{x} \mid x \in B \} $ be the corresponding regular $ (B,\circ) $-stable subgroup of $ \Perm(B) $. Then:
\begin{enumerate}
\item a regular $ (B,\circ) $-stable subgroup $ M $ of $ \Perm(B) $ yields a brace isomorphic to $ \B $ if and only if $ M = \varphi^{-1} N \varphi $ for some $ \varphi \in \Aut(B,\circ) $;
\item we have $ \varphi^{-1} N \varphi = N $ if and only if $ \varphi \in \Aut_{Br}(\B) $, the group of brace automorphisms of $ \B $. 
\end{enumerate}
\end{proposition}
\begin{proof}
\begin{enumerate}
\item First let $ \varphi \in \Aut(B,\circ) $, and define a new binary operation on $ B $ by
\begin{equation} \label{eqn_brace_dot_phi} 
x \cdot_{\varphi} y = \varphi^{-1}(\varphi(x) \cdot \varphi(y)) \mbox{ for all } x,y \in B.
\end{equation}
Then $ \B_{\varphi} = (B,\cdot_{\varphi},\circ) $ is a brace, and $ \varphi : \B \rightarrow \B_{\varphi} $ is a brace isomorphism. The regular $ (B,\circ) $-stable regular subgroup of $ \Perm(B) $ corresponding to $ B_{\varphi} $ is $ N_{\varphi} = \{ \eta^{\varphi}_{x} \mid x \in B \} $, where $ \eta^{\varphi}_{x}(y) = x \cdot_{\varphi} y $ for all $ x,y \in B $. Now we have
\begin{eqnarray*}
\eta^{\varphi}_{x}[y] & = & x \cdot_{\varphi} y \\
& = &  \varphi^{-1}(\varphi(x) \cdot \varphi(y)) \\
& = & \varphi^{-1} \eta_{\varphi(x)}[\varphi(y)] \\
& = & \left( \varphi^{-1} \eta_{\varphi(x)} \varphi \right) [y],
\end{eqnarray*}
and so $ N_{\varphi} = \varphi^{-1} N \varphi $. 

Conversely, suppose that $ M $ is a regular $ (B,\circ) $-stable subgroup of $ \Perm(B) $, let $ \mathfrak{C}= (B,\star, \circ) $ be the brace corresponding to $ M $, and suppose that $ \varphi : \mathfrak{C} \rightarrow \B $ is a brace isomorphism. Then $ \varphi \in \Aut(B,\circ) $ and 
\[\varphi(x \cdot^{\prime} y) = \varphi(x) \cdot \varphi(y), \]
so
\begin{eqnarray*}
x \star y & = & \varphi^{-1}\left( \varphi(x) \cdot \varphi(y)\right) \\
& = & x \cdot_{\varphi} y,
\end{eqnarray*}
where the binary operation $ \cdot_{\varphi} $ is defined as in Equation \eqref{eqn_brace_dot_phi}. Therefore $ M  = \varphi^{-1} N \varphi $ for some $ \varphi \in \Aut(B,\circ) $.
\item First suppose that $ \varphi \in \Aut_{Br}(\B) $. Then we have
\[ x \cdot_{\varphi} y =  \varphi^{-1}(\varphi(x) \cdot \varphi(y)) = x \cdot y \mbox{ for all } x,y \in B, \]
so $ \eta^{\varphi}_{x}  = \eta_{x} $ for all $ x \in B $, and so $ N_{\varphi} = N $. 

Conversely, suppose that $ N_{\varphi} = N $. Then for all $ x \in B $ there exists $ x^{\prime} \in B $ such that $ \eta^{\varphi}_{x} = \eta_{x^{\prime}} $. That is:
\[ x \cdot_{\varphi} y = x^{\prime} \cdot y \mbox{ for all } y \in B. \]
Setting $ y = 1_{B} $ we obtain $ x=x^{\prime} $ immediately. Therefore
\[ \varphi(x) \cdot \varphi(y) = \varphi(x \cdot y) \mbox{ for all } x,y \in B, \]
and so $ \varphi \in \Aut_{Br}(\B) $. 
\end{enumerate}
\end{proof}

As a corollary, we recover \cite[Corollary 2.4]{Ze19}:

\begin{corollary}
A given brace $ \B = (B,\cdot,\circ) $ yields
\[ \frac{ | \Aut(B,\circ) | }{ |\Aut _{Br}(\B)| } \]
distinct regular $ (B,\circ) $-stable subgroups of $ \Perm(B) $.
\end{corollary}

We now return to our original formulation, and consider regular $ G $-stable subgroups of $ \Perm(G) $. 

\begin{definition}
We say that two regular, $G$-stable subgroups $N,M$ of $ \Perm(G) $ are {\em brace equivalent} if they yield isomorphic braces (i.e., braces between which there is a bijection respecting the dot and circle operations). 
\end{definition}

Brace equivalence is an equivalence relation, so we have the notion of a {\em brace class} of regular, $G$-stable subgroups, and the brace classes partition the set of regular $ G $-stable subgroups of $ \Perm(G) $. By Proposition \ref{prop_parameterization}, the brace class of a regular $ G $-stable subgroup $ N $ of $ \Perm(G) $ is $ \{ \varphi^{-1} N \varphi \mid \varphi \in \Aut(G) \} $, and this brace class has size $ | \Aut(G) | / |\Aut_{Br}(\B)| $, where $ \B $ is the brace corresponding to $ N $. 
 
  \begin{example}\label{example_brace_class_lambda_G}
 	Let $G$ be a finite group, and let $N=\lambda(G)$ as in Example \ref{example_brace_trivial}, thereby giving rise to the trivial brace. Any automorphism of $(\lambda(G),\circ)$ will also preserve $\cdot$, hence the brace class containing $\lambda(G)$ is precisely $\{\lambda(G)\}$.
 \end{example} 
 
 \begin{example}\label{example_brace_class_rho_G}
 	Let $G$ be a finite group, and let $N=\rho(G)$ as in Example \ref{example_brace_almost_trivial}, thereby giving rise to the almost trivial brace. Any automorphism of $(\rho(G),\circ)$ will also preserve $\cdot$, hence the brace class containing $\rho(G)$ is precisely $\{\rho(G)\}$.
 \end{example}
 
Now we turn to another natural partition of the set of $ G $-stable regular subgroups of $ \Perm(G) $. Recall from Subsection \ref{subsection_HGS} that each regular $ G $-stable subgroup $ N $ of $ \Perm(G) $ corresponds to a Hopf-Galois structure on a Galois extension of fields $ L/K $ with Galois group $ G $, and that it is possible for two distinct Hopf-Galois structures to involve isomorphic Hopf algebras; this phenomenon has recently been studied in papers such as \cite{KKTU19a}, \cite{KKTU19b}, and \cite{TT19}. In particular, \cite[Theorem 2.2]{KKTU19b} shows that if $ N, M $ are regular $ G $-stable subgroups of $ \Perm(G) $ then $ H_{N} \cong H_{M} $ as Hopf algebras if and only if there is an isomorphism $ \theta : N \rightarrow M $ such that $ \,^{\sigma}\theta(\eta) = \theta(^{\sigma}\eta) $ for all $ \eta \in N $ and $ \sigma \in G $. In this case we say that $ N, M $ are {\em $ G $-isomorphic}. Clearly $ G $-isomorphism is an equivalence relation on the set of regular $ G $-stable subgroups of $ \Perm(G) $, and so we obtain a second partition of this set. 

It is also possible to detect Hopf algebra isomorphisms via regular subgroups of $ \Hol(N) $ that are isomorphic to $ G $: see \cite[Theorem 2.11]{KKTU19b}. However, we feel that this concept is more transparent in the $ \Perm(G) $ setting, and will continue to employ this point of view. 
 
It is natural to ask whether there is a connection between brace equivalence and $ G $-isomorphism of regular $ G $-stable subgroups. Our first observations are that neither implies the other in general:

\begin{example}{\bf (Brace equivalence does not imply $ G $-isomorphism) } \label{example_brace_eq_vs_G_iso}
Let $ L/K $ be a Galois extension with Galois group 
\[ G = \langle \sigma, \tau \mid \sigma^{4}=1, \tau^{2}=\sigma^{2}, \tau \sigma \tau^{-1}= \sigma^{-1} \rangle \cong Q_{8}. \]
It is known \cite[Lemma 2.5]{TT19}  that $ L/K $ admits $ 6 $ Hopf-Galois structures of dihedral type. The corresponding regular subgroups of $ \Perm(G) $ are
\[ D_{s,\lambda} = \langle \lambda(s), \lambda(t)\rho(s) \rangle \quad \mbox{ and } \quad D_{s,\rho} = \langle \rho(s), \lambda(s)\rho(t) \rangle, \]
where in each case $ s,t $ are distinct elements of the set $ \{ \sigma, \tau, \sigma\tau \} $, and the choice of $ t $ does not affect the definition of the subgroups. It is also known \cite[Lemma 3.5]{TT19} that the subgroups described above are pairwise non $ G $-isomorphic. 
\\ \\
We can use Proposition \ref{prop_parameterization} to show that the subgroups $ D_{s,\rho} $ are all brace equivalent. For $ \varphi \in \Aut(G) $ and $ g \in G $ we have
\begin{eqnarray*}
\varphi^{-1} \rho(\sigma) \varphi [g] & = & \varphi^{-1} [ \varphi(g) \sigma^{-1} ] \\
& = & g \varphi(\sigma)^{-1} \\
& = & \rho(\varphi(\sigma)) [g],
\end{eqnarray*}
so $ \varphi^{-1} \rho(\sigma) \varphi = \rho(\varphi(s)) $. Similarly, $ \varphi^{-1} \lambda(\sigma)\rho(\tau) \varphi = \lambda(\varphi(\sigma))\rho(\varphi(\tau))$, and so $ \varphi^{-1} D_{\sigma,\rho} \varphi = D_{\varphi(\sigma),\rho} $. Since there exist automorphisms of $ G $ that send $ \sigma $ to each of $ \sigma, \tau, \sigma\tau $, Proposition \ref{prop_parameterization}  implies that the subgroups $ D_{s,\rho} $ are all brace equivalent. 
\end{example}


This example also shows that the subgroups $ D_{s,\rho} $ exhaust their brace class. Similarly, the subgroups $D_{s,\lambda}$ are brace equivalent, and form a second brace class. We could prove this by the methods employed above, but we shall see a more illuminating proof in Section \ref{sec_opposites}.

\begin{example}{\bf ($ G $-isomorphism does not imply brace equivalence)} \label{example_G_iso_vs_Brace_eq}
Let $ L/K $ be a Galois extension with Galois group
\[ G = \langle \sigma, \tau \mid \sigma^{4}=\tau^{2}=1, \; \tau \sigma \tau^{-1} = \sigma^{-1} \rangle \cong D_{4}. \] 
Let $ \eta = \lambda(\sigma)\rho(\tau) $ and $ \pi = \lambda(\tau) $, and let $ N = \langle \eta,\pi \rangle \subseteq \Perm(G) $. Using the fact that the elements of $ \lambda(G) $ and $ \rho(G) $ commute inside $ \Perm(G) $, we see that $ N \cong G $ and that $ N $ acts regularly on $ G $. In fact, $ N $ is $ G $-isomorphic to $ \lambda(G) $: the map $ \theta : \lambda(G) \rightarrow N $ defined by
\[ \theta( \lambda(\sigma) ) = \eta, \quad \theta( \lambda(\tau) ) = \pi \]
is a $ G $-isomorphism. Therefore $ N $ corresponds to a Hopf-Galois structure on $ L/K $ whose Hopf algebra $ H_{N} $ is isomorphic to $ H_{\lambda} $. However, we have already observed that the brace class of the regular $ G $-stable subgroup $ \lambda(G) $ contains only one element, so $ N $ cannot be brace equivalent to $ \lambda(G) $. 	
\end{example}

However, if two regular $ G $-stable subgroups are $ G $-isomorphic then their elements of their respective brace classes can be arranged into $ G $-isomorphic pairs:

\begin{proposition}
Suppose that $ N,M $ are $ G $-isomorphic regular $ G $-stable subgroups of $ \Perm(G) $. Then, for $ \varphi \in \Aut(G) $, $ N_{\varphi}, M_{\varphi} $ are $ G $-isomorphic. 
\end{proposition}
\begin{proof}
Let $ \theta : N \rightarrow M $ be a $ G $-isomorphism. Define $ \theta_{\varphi} : N_{\varphi} \rightarrow M_{\varphi} $ by
\[ \theta_{\varphi} ( \varphi^{-1} \eta \varphi ) = \varphi^{-1} \theta(\eta) \varphi. \]
Then $ \theta_{\varphi} $ is an isomorphism and for $ \sigma \in G $ we have
\begin{eqnarray*}
\theta_{\varphi} \left( \,^{\sigma} \left( \varphi^{-1} \eta \varphi \right) \right) & = & \theta_{\varphi} ( \varphi^{-1} \left( \,^{\varphi(\sigma)} \eta \right) \varphi ) \\
& = &  \varphi^{-1} \theta \left( \,^{\varphi(\sigma)} \eta \right) \varphi \\
& = &  \varphi^{-1}  \,^{\varphi(\sigma)} \theta \left( \eta \right) \varphi \\
& = &  \,^{\sigma} \left( \varphi^{-1}   \theta \left( \eta \right) \varphi \right) \\
& = &  \,^{\sigma} \theta_{\varphi} \left( \varphi^{-1} \eta \varphi \right).
\end{eqnarray*}
Hence $ N_{\varphi}, M_{\varphi} $ are $ G $-isomorphic. 
\end{proof}

\section{Inner automorphisms and $ \rho $-conjugate subgroups} \label{sec_rho_conjugate}

In this section we assume that $ G $ is nonabelian and explore the consequences of conjugating a $ G $-stable regular subgroup $ N $ of $ \Perm(G) $ by an inner automorphism of $ G $. 

\begin{proposition} \label{prop_inner_rho_conjugate}
Let $ G $ be a nonabelian group and let $ N $ be a regular $ G $-stable subgroup of $ \Perm(G) $. Let $ \sigma \in G $, and let $ C(\sigma) $ denote the inner automorphism of $ G $ arising from $ \sigma $. Then:
\begin{enumerate}
\item $ C(\sigma)N C(\sigma)^{-1}= \rho(\sigma) N \rho(\sigma)^{-1} $, where $ \rho : G \rightarrow \Perm(G) $ is the the right regular representation of $ G $;
\item the subgroups $ N $ and $ C(\sigma)N C(\sigma)^{-1}$ are $ G $-isomorphic. 
\end{enumerate}
\end{proposition}
\begin{proof}
\begin{enumerate}
\item We may write $ C(\sigma)= \rho(\sigma) \lambda(\sigma) $, and so 
\[ C(\sigma)N C(\sigma)^{-1}= \rho(\sigma) \lambda(\sigma) N \lambda(\sigma)^{-1} \rho(\sigma)^{-1}. \]
But $ N $ is $ G $-stable, so $ \lambda(\sigma) N \lambda(\sigma)^{-1} = N $, and so
\[ C(\sigma)N C(\sigma)^{-1}= \rho(\sigma) N \rho(\sigma)^{-1}, \]
as claimed. 
\item Consider the isomorphism $ \theta : N \rightarrow \rho(\sigma) N \rho(\sigma)^{-1} $ defined by $ \theta(\eta) = \rho(\sigma) \eta \rho(\sigma)^{-1} $ for all $ \eta \in N $. Using the fact that elements of $ \lambda(G) $ and $ \rho(G) $ commute inside $ \Perm(G) $ we have $  \theta(^{\tau}\eta)=\,^{\tau} \theta(\eta) $ for all $ \eta \in N $ and $ \tau \in G $, so $ \theta $ is a $ G $-isomorphism. 
\end{enumerate}
\end{proof}

We shall say that two regular $ G $-stable subgroups $ N,M $ of $ \Perm(G) $ are {\em $ \rho $-conjugate} if $ M = \rho(\sigma) N \rho(\sigma)^{-1} $ for some $ \sigma \in G $. This concept also appears in \cite[Example 2.7]{KKTU19a}. We record some immediate corollaries of Proposition \ref{prop_inner_rho_conjugate}:

\begin{corollary} \label{cor_rho_brace_equivalent}
If $ N,M $ are $ \rho $-conjugate regular $ G $-stable subgroups of $ \Perm(G) $ then they are brace equivalent. 
\end{corollary}

\begin{corollary}
If  $ G $ has only inner automorphisms then, for regular $ G $-stable subgroups of $ \Perm(G) $, brace equivalence implies $ G $-isomorphism. 
\end{corollary}

\begin{example}
Let $ G \cong D_{3} $ as in Example \ref{example_D_3_HGS}, and consider the regular $ G $-stable subgroups  $ N_{c} = \langle \lambda(\sigma),\rho(\sigma^{c}\tau) \rangle $ of $ \Perm(G) $ constructed there. It is not hard to see that these subgroups are $ \rho $-conjugate, which implies that they are brace equivalent, as stated at the end of Example \ref{example_D_3_brace_equivalence} and $ G $-isomorphic. 
\end{example}

\section{Abelian endomorphisms} \label{sec_fpf}

An endomorphism $ \psi : G \rightarrow G $ is called {\em abelian} if $ \psi(\sigma \tau) = \psi(\tau \sigma) $ for all $ \sigma,\tau \in G $, and {\em fixed-point-free} if $ \psi(\sigma)=\sigma $ only when $ \sigma=1_{G} $. In \cite{Ch13} Childs shows that, given a Galois extension of fields $ L/K $ with nonabelian Galois group $ G $, abelian fixed-point-free endomorphisms can be used to construct families of regular $ G $-stable subgroups of $ \Perm(G) $ that are isomorphic to $ G $. In \cite{Koc21} the first named author generalizes this construction by removing the fixed-point-free hypothesis; a consequence of this is that resulting subgroups are not necessarily isomorphic to $ G $. In this section we study the braces corresponding to subgroups arising from abelian endomorphisms.


First we summarize the results of \cite{Koc21}. Suppose that $ \psi : G \rightarrow G $ is an abelian endomorphism, and define a map $ \alpha_{\psi} : G \rightarrow \Perm(G) $ by $ \alpha_{\psi}(\sigma) = \lambda(\sigma) C(\psi(\sigma^{-1})) $ for all $ \sigma \in G $, where $ C(\psi(\sigma)) $ denotes the inner automorphism arising from $ \psi(\sigma) $. It is easy to see that $ \alpha_{\psi} $ is a homomorphism, so that $ N_{\psi} = \alpha_{\psi}(G) $ is a subgroup of $ \Perm(G) $, and it can be shown that it is regular and $ G $-stable \cite[Theorem 3.1]{Koc21}. 


Now we study the relationships between braces corresponding to regular $ G $-stable subgroups arising via this construction. Write $ \mathrm{Ab}(G) $ for the set of abelian endomorphisms of $ G $. 

\begin{proposition}
If $ \psi \in \mathrm{Ab}(G) $ and $ \varphi \in \Aut(G) $ then $ \varphi^{-1}\psi \varphi \in \mathrm{Ab}(G) $.
\end{proposition}
\begin{proof}
It is clear that $ \varphi^{-1}\psi \varphi $ is an endomorphism of $ G $; we need to show that it is abelian. For $ \sigma,\tau \in G $ we have
\begin{eqnarray*}
\varphi^{-1}\psi \varphi (\sigma\tau) & = & \varphi^{-1}\psi (\varphi(\sigma) \varphi(\tau)) \mbox{ ($\varphi $ is an automorphism) } \\
& = & \varphi^{-1}\psi (\varphi(\tau)\varphi(\sigma)) \mbox{ ($\psi $ is abelian) } \\
& = & \varphi^{-1}\psi \varphi(\tau\sigma).
\end{eqnarray*}
Therefore $ \varphi^{-1}\psi \varphi $ is abelian, and hence $ \varphi^{-1}\psi \varphi \in \mathrm{Ab}(G) $. 
\end{proof}

\begin{proposition} \label{prop_fpf_subgroups}
If $ \psi \in \mathrm{Ab}(G) $ and $ \varphi \in \Aut(G) $ then $ N_{\varphi^{-1}\psi\varphi} = \varphi^{-1} N_{\psi} \varphi $. 
\end{proposition}
\begin{proof}
Let $ \sigma \in G $. Then for all $ \tau\in G $ we have
\begin{eqnarray*}
\varphi^{-1}\alpha_{\psi}(\sigma)\varphi [\tau] & = & \varphi^{-1} \lambda(\sigma) C(\psi(\sigma^{-1})) \varphi [\tau] \\
& = & \varphi^{-1} ( \sigma \psi(\sigma^{-1}) \varphi(\tau) \psi(\sigma) ) \\
& = & \varphi^{-1}(\sigma) \varphi^{-1}(\psi(\sigma^{-1})) \tau \varphi^{-1}(\psi(\sigma) ) \\
& = & \lambda(\varphi^{-1}(\sigma)) C(\varphi^{-1}(\psi(\sigma^{-1}))) [\tau] \\
& = & \lambda(\varphi^{-1}(\sigma)) C(\varphi^{-1}\psi\varphi(\varphi^{-1}(\sigma^{-1}))) [\tau] \\
& = & \alpha_{\varphi^{-1}\psi\varphi}(\varphi^{-1}(\sigma)) [\tau]. 
\end{eqnarray*}
Therefore $ \varphi^{-1}\alpha_{\psi}(\sigma)\varphi  = \alpha_{\varphi^{-1}\psi\varphi}(\varphi^{-1}(\sigma)) $ for all $ \sigma \in G $, and so $ \varphi^{-1}N_{\psi}\varphi = N_{\varphi^{-1}\psi\varphi} $. 
\end{proof}

To ease notation, if $ \psi \in \mathrm{Ab}(G) $ then we write $ \B_{\psi} $ rather than $ \B_{N_{\psi}} $ for the brace corresponding to $ N_{\psi} $, however, we caution the reader that two different elements of $ \mathrm{Ab}(G) $ can yield the same subgroup, so $ \B_{\psi} =  \B_{\psi'} $ does not imply that $ \psi = \psi' $. We also note that this construction may yield a different brace from the construction in \cite{KST}.

\begin{proposition}
If $ \psi \in \mathrm{Ab}(G) $ and $ \varphi \in \Aut(G) $ then $ \B_{\psi} \cong \B_{\varphi^{-1}\psi\varphi} $. Furthermore, if $ \B_{\psi} \cong\B_{\psi^{\prime}} $ for 
some $ \psi^{\prime} \in \mathrm{Ab}(G) $ then there exists $ \varphi \in \Aut(G) $ such that $ \psi^{\prime} = \varphi^{-1}\psi\varphi $. 
\end{proposition}
\begin{proof}
By Proposition \ref{prop_fpf_subgroups} we have $ N_{\varphi^{-1}\psi\varphi} = \varphi^{-1}N_{\psi}\varphi $; hence $ \B_{\psi} \cong \B_{\varphi^{-1}\psi\varphi} $ by Proposition \ref{prop_parameterization}.  If $ \B_{\psi} \cong\B_{\psi^{\prime}} $ for some $ \psi^{\prime} \in \mathrm{Ab}(G) $ then by Proposition \ref{prop_parameterization} $ N_{\psi^{\prime}} = \varphi^{-1} N_{\psi} \varphi $ for some $ \varphi \in \Aut(G) $; by Proposition \ref{prop_fpf_subgroups} we have $ \varphi^{-1} N_{\psi} \varphi = N_{\varphi^{-1}\psi\varphi} $, and so $ N_{\psi^{\prime}}  = N_{\varphi^{-1}\psi\varphi}  $.
\end{proof}

\begin{proposition} 
If $ \psi \in \mathrm{Ab}(G) $ and $ N $ is a regular $ G $-stable subgroup of $ \Perm(G) $ that is brace equivalent to $ N_{\psi} $ then $ N = \alpha_{\psi^{\prime}}(G) $ for some $ \psi^{\prime} \in \mathrm{Ab}(G) $. 
\end{proposition}
\begin{proof}
Since $ N $ is brace equivalent to $ N_{\psi} $ there exists $ \varphi \in \Aut(G) $ such that $ N = \varphi^{-1}N_{\psi}\varphi $. Applying Proposition \ref{prop_fpf_subgroups} we have $ N = N_{\varphi^{-1}\psi\varphi} $, so $ N = \alpha_{\psi^{\prime}}(G) $ with $ \psi^{\prime} = \varphi^{-1}\psi\varphi $. 
\end{proof}

If we impose the additional assumption that $ \psi $ is fixed-point-free (as in \cite{Ch13}) then \cite[Section 4]{Koc21} shows that $ N_{\psi} \cong G $ and that there exists another fixed-point-free abelian endomorphism $ \Psi $ such that $ N_{\psi} = \{ \lambda(\sigma) \rho(\Psi(\sigma)) \mid \sigma \in G \} $; following \cite{Ch13} we then see that the isomorphism $ \theta: \lambda(G) \rightarrow N_{\psi} $ defined by $ \theta(\lambda(\sigma)) = \lambda(\sigma) \rho(\Psi(\sigma)) $ is a $ G $-isomorphism. Thus we have:

\begin{corollary} \label{cor_fpf_G_iso}
If $ \psi \in \mathrm{Ab}(G) $ is fixed-point free and $ N $ is a regular $ G $-stable subgroup of $ \Perm(G) $ that is brace equivalent to $ N_{\psi} $ then $ N $ is $ G $-isomorphic to $ \lambda(G) $. 
\end{corollary}

\section{$ \lambda $-points and $ \rho $-points}

The prototypical examples of regular $ G $-stable subgroups of $ \Perm(G) $ are the subgroups $ \lambda(G) $ and $ \rho(G) $. A general regular $ G $-stable subgroup $ N $ may intersect nontrivially with one or both of these; in this section we show that studying these intersections can yield useful information about the brace and Hopf-Galois structure that correspond to $ N $. 

\begin{definition}
Let $ N $ a regular $ G $-stable subgroup of $ \Perm(G) $. The {\em $ \lambda $-points} of $ N $ are the elements of the set $ \Lambda_{N} = N \cap \lambda(G) $. The {\em $ \rho $-points} of $ N $ are the elements of the set $ P_{N} = N \cap \rho(G) $. 
\end{definition}

It is clear that $ \Lambda_{N} $ and $ P_{N} $ are both subgroups of $ N $.

\begin{example}
Let $ L/K $ be a Galois extension with Galois group $ G \cong Q_{8} $, as in Example \ref{example_brace_eq_vs_G_iso}, and consider the regular $ G $-stable subgroups $ D_{s,\lambda} $ and $ D_{s,\rho} $ constructed there. Then the $ \lambda $-points of $ D_{s,\lambda} $ are $ \lambda(1), \lambda(s), \lambda(s^{2}), \lambda(s^{3}) $, and the $ \rho $-points of $ D_{s,\lambda} $ are $ \rho(1) $ and $ \rho(s^{2}) $, since $ s^{2} \in Z(G) $. The results for $ D_{s,\rho} $ are analogous. 
\end{example}

First we study the behaviour of $ \lambda $-points and $ \rho $-points with respect to brace equivalence:

\begin{proposition} \label{prop_lambda_points_rho_points}
Let $ N,M $ be regular $ G $-stable subgroups of $ \Perm(G) $ and suppose that $ N,M $ are brace equivalent. Then:
\begin{enumerate}
\item\label{item_lambda_points_iso} $ \Lambda_{N} \cong \Lambda_{M} $;
\item\label{item_rho_points_iso}  $ P_{N} \cong P_{M} $. 
\end{enumerate}
\end{proposition}
\begin{proof}
Since $ N, M $ are brace equivalent, there exists $ \varphi \in \Aut(G) $ such that $ M = \varphi^{-1}N\varphi $. To prove \eqref{item_lambda_points_iso}, define $ \theta : \Lambda_{N} \rightarrow M $ by $ \theta(\lambda(\sigma)) = \varphi^{-1}(\lambda(\sigma)) \varphi $. Then for all $ \tau \in G $ we have
\[ \theta(\lambda(\sigma))[\tau] = \varphi^{-1}(\lambda(\sigma)) \varphi [\tau] = \varphi^{-1}( \sigma \varphi(\tau) ) = \varphi^{-1}( \sigma ) \tau = \lambda( \varphi^{-1}( \sigma ) ) [\tau]. \]
Hence $ \theta $ is actually a map from $ \Lambda_{N} $ to $ \Lambda_{M} $, which is clearly an isomorphism. The proof of \eqref{item_rho_points_iso} is similar. 
\end{proof}

Proposition \ref{prop_lambda_points_rho_points} provides a useful necessary condition for two regular $ G $-stable subgroups to be brace equivalent, which we shall apply in Section \ref{sec_brace_classifications}.
\\ \\
In fact, the isomorphisms established in Proposition \ref{prop_lambda_points_rho_points} are $ G $-isomorphisms. More generally, $ \rho $-points interact well with $ G $-isomorphism:

\begin{proposition} \label{prop_rho_points_isomorphism}
Let $ N, M $ be regular $ G $-stable subgroups of $ \Perm(G) $ and suppose that $ \theta : N \rightarrow M $ is a $ G $-isomorphism. Then $ P_{N} = \theta( P_{M} ) $. 
\end{proposition}
\begin{proof}
We may characterize $ \rho(G) $ as the centralizer of $ \lambda(G) $ in $ \Perm(G) $: thus a permutation $ \eta \in \Perm(G) $ lies in $ \rho(G) $ if and only if $ \,^{\sigma}\eta = \eta $ for all $ \sigma \in G $. Now let $ \eta \in P_{N} $; then for all $ \sigma \in G $ we have
\begin{eqnarray*}
\,^{\sigma} \theta(\eta) & = & \theta(^{\sigma}\eta) \mbox{ since $ \theta $ is a $ G $-isomorphism } \\
& = & \theta(\eta) \mbox{ since $ \eta \in P_{N} $. }
\end{eqnarray*}
Hence $ \theta(\eta) \in P_{M} $. Reversing the roles of $ N, M $ yields the result. 
\end{proof}

Proposition \ref{prop_lambda_points_rho_points} provides a useful necessary condition for two regular $ G $-stable subgroups to be  $ G $-isomorphic. However, the analogous result for $ \lambda $-points is not true:

\begin{example}
Let $ L/K $ be a Galois extension with Galois group $ G \cong D_{4} $, as in Example \ref{example_G_iso_vs_Brace_eq}, and recall from that example that $ N = \langle \eta,\pi \rangle $, with $ \eta = \lambda(\sigma)\rho(\tau) $ and $ \pi = \lambda(\tau) $, is a $ G $-stable regular subgroup of $ \Perm(G) $ that is $ G $-isomorphic to $ \lambda(G) $. Obviously every element of $ \lambda(G) $ is a $ \lambda $-point, but the $ \lambda $-points of $ N $ are $ \lambda(1), \lambda(\tau), \lambda(\sigma^{2}), \lambda(\sigma^{2}\tau) $. 
\end{example}

\section{Opposite braces} \label{sec_opposites} 

In Greither and Paregis's original paper characterizing Hopf-Galois structures on separable field extensions \cite{GP87} they observe that if $ N $ is a regular $ G $-stable subgroup of $ \Perm(G) $, then so too is $ N^{opp} = \mathrm{Cent}_{\Perm(G)}(N) $; this construction, and the corresponding Hopf-Galois structures on a Galois extension with Galois group $ G $, have subsequently been studied in, for example, \cite{Ko07}, \cite{Tr18}, and \cite{KT20}. The notation $ N^{opp} $ reflects the fact that this subgroups can be naturally identified with the opposite group of $ N $; in \cite{KT20} the authors referred to the Hopf-Galois structure corresponding to $ N^{opp} $ as the {\em opposite} of the one corresponding to $ N $. We showed that this construction leads naturally to the notion of the {\em opposite} of a brace, as follows: given a brace $ \B = (B,\cdot,\circ) $, we define a new binary operation on $ B $ by $ x \cdot^{\prime} y = y \cdot x $ for all $ x,y \in B $. Then $ \B^{opp} := (B,\cdot^{\prime},\circ) $ is a brace, called the {\em opposite} of the brace $ \B $. This concept is also developed independently by Rump \cite{Ru19}. If $ N $ is a regular $ G $-stable subgroup of $ \Perm(G) $, with corresponding brace $ \B_{N} $, then the brace corresponding to the opposite subgroup $ N^{opp} $ is then $ \B_{N^{opp}} = \left(\B_{N}\right)^{opp} $. If $ \B = (B,\cdot,\circ) $ is a brace and $ N $ is the regular $ (B,\circ) $-stable subgroup of $ \Perm(B) $ arising from the dot operation in $ B $ then the subgroup arising from the dot operation in $ \B^{opp} $ is $ N^{opp} $. 


The notion of opposite extends to brace classes:

\begin{proposition}
Let $ N $ be a regular $ G $-stable subgroup of $ \Perm(G) $, and let $ N^{opp} $ be the opposite subgroup to $ N $. Then for each $ \varphi \in \Aut(G) $ we have $ \left( \varphi^{-1} N \varphi \right)^{opp} = \varphi^{-1} N^{opp} \varphi $. 
\end{proposition}
\begin{proof}
Let $ \eta^{\prime} \in N^{opp} $, so that $ \varphi^{-1} \eta^{\prime} \varphi \in \varphi^{-1} N^{opp} \varphi $. Then for all $ \eta \in N $ we have 
\begin{eqnarray*}
\left( \varphi^{-1} \eta^{\prime}  \varphi \right) \left( \varphi^{-1} \eta \varphi \right) & = & \left( \varphi^{-1} \eta^{\prime} \eta \varphi \right) \\
& = & \left( \varphi^{-1} \eta \eta^{\prime} \varphi \right) \\
& = & \left( \varphi^{-1} \eta  \varphi \right) \left( \varphi^{-1} \eta^{\prime} \varphi \right).
\end{eqnarray*} 
Hence $ \varphi^{-1} N^{opp} \varphi \subseteq \left( \varphi^{-1} N \varphi \right)^{opp} $. But these groups have equal order, so in fact they are equal. 
\end{proof}

\begin{corollary}
The brace class of $ N^{opp} $ consists precisely of the opposites of the subgroups in the brace class of $ N $. In particular, these brace classes are of equal size. 
\end{corollary}

\begin{example}
Let $ L/K $ be a Galois extension with Galois group $ G = \langle \sigma, \tau \rangle \cong Q_{8} $, as in Example \ref{example_brace_eq_vs_G_iso}, and consider the regular $ G $-stable subgroups $ D_{s,\lambda} $ and $ D_{s,\rho} $ described in that example. We saw there that the subgroups $ D_{s,\rho} $ are brace equivalent and exhaust their brace class. It is routine to verify that $ D_{s,\rho}^{opp} = D_{s,\lambda} $ for each $ s $; thus the subgroups $ D_{s,\lambda} $ are brace equivalent and exhaust their brace class, as stated at the end of that example. 
\end{example}

\begin{corollary} \label{cor_opposite_rho_conjugate}
Let $ N, M $ be $ G $-stable regular subgroups of $ \Perm(G) $, and suppose that $ N,M $ are $ \rho $-conjugate. Then $ N^{opp}, M^{opp} $ are $ \rho $-conjugate.
\end{corollary}

As pointed out in \cite[Section 6]{KT20}, it is possible for $\B_N$ and $\B_{N^{opp}}$ to be isomorphic; when this occurs, the brace classes of $ N $ and $ N^{opp} $ coincide.  

On the other hand, if $ N,M $ are regular $ G $-stable subgroups of $ \Perm(G) $ that are $ G $-isomorphic, it does not necessarily follow that $ N^{opp} $ and $ M^{opp} $ are $ G $-isomorphic:

\begin{example}
Let $ L/K $ be a Galois extension with Galois group $ G \cong D_{4} $, as in Example \ref{example_G_iso_vs_Brace_eq}, and recall from that example that $ N = \langle \eta,\pi \rangle $, with $ \eta = \lambda(\sigma)\rho(\tau) $ and $ \pi = \lambda(\tau) $, is a $ G $-stable regular subgroup of $ \Perm(G) $ that is $ G $-isomorphic to $ \lambda(G) $. However, we have $ \lambda(G)^{opp} = \rho(G) $, and no other regular $ G $-stable subgroup of $ \Perm(G) $ can be $ G $-isomorphic to $ \rho(G) $. Therefore $ N^{opp} $ is not $ G $-isomorphic to $ \lambda(G)^{opp} $. 	
\end{example}

\section{Braces of order $ pq $ and Hopf-Galois structures of degree $ pq $} \label{sec_brace_classifications}

Let $ p,q $ be prime numbers with $ p>q $. In \cite{By04} Byott classifies the Hopf-Galois structures admitted by Galois extensions of fields of degree $ pq $; building upon these results, the braces of order $ pq $ are classified in \cite{AB20} and \cite[Subsection 2.9]{CCDC20}. 

In this section we consider in turn each of the isomorphically distinct braces $ \B = (B,\cdot, \circ) $ of order $ pq $. Writing $ N = (B,\cdot) $ and $ G = (B,\circ) $, we recover all of the $ G $-stable regular subgroups of $ \Perm(G) $ that are isomorphic to $ N $, and hence the Hopf-Galois structures of type $ N $ on a Galois extension of fields with Galois group $ G $. We use the tools developed in the earlier sections to arrange these subgroups into $ G $-isomorphism classes, which is equivalent to determining the Hopf algebra isomorphism classes of the Hopf algebras giving the Hopf-Galois structures. The classifications of braces in \cite{AB20} and \cite[Subsection 2.9]{CCDC20} are organized by fixing a presentation of the dot group and allowing the circle group to vary. Here it is more convenient to reverse this organization; where necessary, we state maps that reconcile our descriptions with those in loc. cit. 

If $ p \not \equiv 1 \pmod{q} $ then $ N $ and $ G $ must both be isomorphic to $ C $, the cyclic group of order $ pq $; if $ p \equiv 1 \pmod{q} $ then each of these groups is isomorphic to either $ C $ or to $ M $, the metacyclic group of order $ pq $. We describe both these groups via two generators $ \sigma, \tau $, using the underlying set
\[ B = \{ \sigma^{i}\tau^{j} \mid 0 \leq i \leq p-1, \; 0 \leq j \leq q-1 \}, \]
where $ \sigma $ has order $ p $ and $ \tau $ has order $ q $. To obtain $ C $ we impose the relation $ \tau\sigma\tau^{-1} =\sigma $; to obtain $ M $ we fix an integer $ g $ whose multiplicative order modulo $ p $ is $ q $ and impose the relation $ \tau\sigma\tau^{-1} =\sigma^{g} $. We also fix notation for the automorphisms of these groups: we have $ \Aut(C) = \{ \varphi_{t} \mid \gcd(t,pq)=1 \} $, where $ \varphi_{t}(\sigma)=\sigma^{t} $ and $ \varphi_{t}(\tau) = \tau^{t} $, and $ \Aut(M) = \langle \phi, \psi \rangle $, where $ \phi(\sigma\tau) = \sigma^{d}\tau $ with $ d $ a primitive root modulo $ p $ and $ \phi(\tau) = \tau $, and $ \psi(\sigma) = \sigma $ and $ \psi(\tau) = \sigma \tau$. 
\\ \\
\noindent \textbf{The case} $ \mathbf{G \cong N \cong C: }$
Up to isomorphism, there is a unique brace $ \B $ with $ N \cong G \cong C $, which is the trivial brace for $ C $ \cite[Proposition 3.1]{AB20}. The regular $ G $-stable subgroup of $ \Perm(G) $ obtained from the dot operation in $ \B $ is simply $ \rho(G) $. Moreover, every automorphism of $ G $ is automatically an automorphism of $ \B $; hence $ \rho(G) $ is the unique regular $G$-stable subgroup of $\Perm(G)$ that is isomorphic to $ N $ \cite[(4.1)]{By04}. There are no $ G $-isomorphism questions to consider in this case. 
\\ \\
As mentioned above, if $ p\not\equiv 1 \pmod{q}$ then every group of order $pq$ is cyclic, and so this case is the only one that can occur. For the remainder of this section we assume that $p\equiv 1 \pmod{q}$. 
\\ \\
\noindent  \textbf{The case} $ \mathbf{G \cong C, \; N \cong M: }$
There are two isomorphically distinct braces with these properties. The first is the brace $ A_{q} $ constructed in  \cite[part (iii) of the second bullet point of the main theorem]{AB20}. In our notation this takes the form $ \B = (B,\cdot,\circ) $, where 
\[ \begin{array}{lll}
\sigma^{i}\tau^{j} \cdot \sigma^{k}\tau^{\ell} & = & \sigma^{i+kg^{j}} \tau^{j+\ell} \\
\sigma^{i}\tau^{j} \circ \sigma^{k}\tau^{\ell} & = & \sigma^{i+k} \tau^{j+\ell}.
\end{array} \]
The second brace with these properties is constructed in \cite[part (ii) of the second bullet point of the main theorem]{AB20}. We obtain it by taking the opposite to the brace $ \B $ above: we have $ \B^{opp} = (B,\cdot',\circ) $, where
\[ \begin{array}{lll}
\sigma^{i}\tau^{j} \cdot' \sigma^{k}\tau^{\ell} & = & \sigma^{k+ig^{\ell}} \tau^{j+\ell} \\
\sigma^{i}\tau^{j} \circ \sigma^{k}\tau^{\ell} & = & \sigma^{i+k} \tau^{j+\ell}.
\end{array} \]
To reconcile this description with the one given in loc. cit. we apply the map $ \sigma^{i}\tau^{j} \mapsto \tau^{j}\sigma^{i} $.

The regular $ G $-stable subgroup  of $ \Perm(G) $ obtained from the dot operation in $ \B $ is $ N = \langle \eta, \pi \rangle $, where
\[ \begin{array}{l}
\eta(\sigma^{k}\tau^{\ell}) = \sigma \cdot \sigma^{k}\tau^{\ell} = \sigma^{k+1}\tau^{\ell} \\
\pi(\sigma^{k}\tau^{\ell}) = \tau \cdot \sigma^{k}\tau^{\ell} = \sigma^{kg}\tau^{\ell+1}.
\end{array} \]
We determine all of the regular $ G $-stable regular subgroups of $ \Perm(G) $ that are brace equivalent to $ N $. For each $ t $ coprime to $ pq $ we conjugate the generators of $ N $ by the automorphism $ \varphi_{t} $:
\[ \begin{array}{l}
\varphi_{t}^{-1}\eta\varphi_{t}(\sigma^{k}\tau^{\ell}) = \sigma^{k+t^{-1}}\tau^{\ell} = \eta^{t^{-1}}(\sigma^{k}\tau^{\ell}) \\
\varphi_{t}^{-1}\pi\varphi_{t}(\sigma^{k}\tau^{\ell}) = \sigma^{kg}\tau^{\ell+t^{-1}} = \pi_{t}(\sigma^{k}\tau^{\ell}),
\end{array} \]
say (with $ \pi = \pi_{1} $). The automorphism $ \varphi_{t} $ respects $ \cdot $ if and only if $ t \equiv 1 \pmod{q} $, and so by Proposition \ref{prop_parameterization} the subgroups $ N_{t} = \langle \eta, \pi_{t} \rangle $ are a family of $ q-1 $ regular $ G $-stable subgroups of $ \Perm(G) $. This is the family of subgroups described in \cite[Lemma 5.2, Equation (5.8)]{By04}.
\\ \\
The regular $ G $-stable subgroup  of $ \Perm(G) $ obtained from the dot operation in $ \B^{opp} $ is $ N^{opp} = \langle \eta', \pi' \rangle $, where
\[ \begin{array}{l}
\eta'(\sigma^{k}\tau^{\ell}) = \sigma \cdot^{\prime} \sigma^{k}\tau^{\ell} = \sigma^{k+g^{\ell}}\tau^{\ell} \\
\pi'(\sigma^{k}\tau^{\ell}) = \tau \cdot^{\prime} \sigma^{k}\tau^{\ell} = \sigma^{k}\tau^{\ell+1}.
\end{array} \]
We determine all of the regular $ G $-stable regular subgroups of $ \Perm(G) $ that are brace equivalent to $ N^{opp} $. Proceeding as above we have
\[ \begin{array}{l}
\varphi_{t}^{-1}\eta'\varphi_{t}(\sigma^{k}\tau^{\ell}) = \sigma^{k+t^{-1}g^{t\ell}}\tau^{\ell} = \eta'_{t}(\sigma^{k}\tau^{\ell}) \\
\varphi_{t}^{-1}\pi'\varphi_{t}(\sigma^{k}\tau^{\ell}) = \sigma^{k}\tau^{\ell+t^{-1}} = (\pi')^{t^{-1}}(\sigma^{k}\tau^{\ell}),
\end{array} \]
say (with $ \eta' = \eta'_{1} $). The automorphism $ \varphi_{t} $ respects $ \cdot' $ if and only if $ t \equiv 1 \pmod{q} $, and so the subgroups $ N^{opp}_{t} = \langle \eta', \pi'_{t} \rangle $ are a family of $ q-1 $ regular $ G $-stable subgroups of $ \Perm(G) $. This is the family of subgroups described in \cite[Lemma 5.1, Equation (5.3)]{By04}. The subgroups $ N_{t} $ and $ N'_{t} $ account for all of the $ G $-stable regular subgroups of $ \Perm(G) $ that are isomorphic to $ M $.

\begin{proposition}
The subgroups $ N_{t} $ are mutually $ G $-isomorphic. The subgroups $ N_{t}^{opp} $ are pairwise non $ G $-isomorphic, and none are $ G $-isomorphic to any of the subgroups $ N_{t} $. 
\end{proposition}
\begin{proof}
For $ 1 \leq t \leq q-1 $ the action of $ G $ on the subgroup $ N_{t} $ is given by 
\[ \begin{array}{ll}
\,^{\sigma}\eta = \eta, & \,^{\tau}\eta = \eta, \\
\,^{\sigma} \pi_{t} = \eta^{1-g} \pi_{t},  & \,^{\tau}\pi_{t} = \pi_{t}. 
\end{array} \]
Hence the natural isomorphism $ N_{1} \rightarrow N_{t} $ defined by $ \eta^{i}\pi_{1}^{j} \mapsto \eta_{i}\pi_{t}^{j} $ is a $ G $-isomorphism, and so the subgroups $ N_{t} $ are mutually $ G $-isomorphic. 

On the other hand for $ 1 \leq t \leq q-1 $ the action of $ G $ on the subgroup $ N_{t}^{opp} $ is given by
\[ \begin{array}{ll}
\,^{\sigma}\eta_{t}' = \eta_{t}', & \,^{\tau}\eta_{t}' = (\eta_{t}')^{g^{-t}}, \\
\,^{\sigma} \pi' = \pi',  & \,^{\tau}\pi' = \pi'. 
\end{array} \]
Now if $ \theta: N^{opp}_{t_{1}} \rightarrow N^{opp}_{t_{2}} $ is a $ G $-isomorphism then $ \theta(\eta_{t_{1}}) = \eta_{t_{2}}^{v} $ for some $ v = 1, \ldots ,p-1 $, since $ \eta_{t_{1}} $ has order $ p $. Taking the $ G $ action into account, we have
\[ \begin{array}{rll}
& \theta( \,^{\tau} \eta_{t_{1}} ) = \theta( \eta_{t_{1}}^{g^{-t_{1}}} ) = \eta_{t_{2}}^{vg^{-t_{1}}} \\
\mbox{ and } & \,^{\tau} \theta( \eta_{t_{1}} ) = \,^{\tau} \eta_{t_{2}}^{v} = \eta_{t_{2}}^{vg^{-t_{2}}}.
\end{array} \]
Since $ v \not \equiv 0 \pmod{p} $, this implies that $ g^{-t_{1}} \equiv g^{-t_{2}} \pmod{p} $, and since $ g $ has order $ q $ modulo $ p $ this implies that $ t_{1}=t_{2} $.

Finally note that since $ G $ is abelian the notions of $ \lambda $-point and $ \rho $-point coincide. The subgroups of the form $ N_{t} $ have $ p $ $ \rho $-points, namely the elements $ \eta^{i} $. The subgroups of the form $ N_{t}^{opp} $ have $ q $ $ \rho $-points, namely the elements $ (\pi')^{j} $. By Proposition \ref{prop_rho_points_isomorphism} no subgroup of the form $ N_{t} $ can be $ G $-isomorphic to a subgroup of the form $ N_{t}^{opp} $. Therefore the $ G $-isomorphisms amongst these subgroups are as in the statement of the proposition.
\end{proof}

\noindent  \textbf{The case $\mathbf{G \cong M, \; N \cong C: }$} 
Up to isomorphism there is a unique brace with these properties, which is constructed in  \cite[part (ii) of the first bullet point of the main theorem]{AB20}. In our notation this takes the form $ \B = (B,\cdot,\circ) $, where 
\[ \begin{array}{lll}
\sigma^{i}\tau^{j} \cdot \sigma^{k}\tau^{\ell} & = & \sigma^{i+k} \tau^{j+\ell} \\
\sigma^{i}\tau^{j} \circ \sigma^{k}\tau^{\ell} & = & \sigma^{i+kg^j} \tau^{j+\ell}.
\end{array} \]
The regular $ G $-stable subgroup of $ \Perm(G) $ obtained from the dot operation in $ \B $ is $ N = \langle \eta, \pi \rangle $, where
\[ \begin{array}{l}
\eta(\sigma^{k}\tau^{\ell}) = \sigma \cdot \sigma^{k}\tau^{\ell} = \sigma^{k+1}\tau^{\ell} \\
\pi(\sigma^{k}\tau^{\ell}) = \tau \cdot \sigma^{k}\tau^{\ell} = \sigma^{k}\tau^{\ell+1}.
\end{array} \]
We determine all of the regular $ G $-stable regular subgroups of $ \Perm(G) $ that are brace equivalent to $ N $.  The automorphism $ \phi $ respects $ \cdot $, but the automorphism $ \psi $ does not; thus the subgroups $ N_{s} = \psi^{-s} N \psi^{s} $, with $ 0, \leq s \leq p-1 $, form a family of $ p $ regular $G$-stable subgroups of $\Perm(G)$ that are isomorphic to $ N $. This is the family described in \cite[Lemma 4.1, Equation (4.3)]{By04}, and accounts for all of the $ G $-stable regular subgroups of $ \Perm(G) $ that are isomorphic to $ C $. 

\begin{proposition} \label{prop_N_c_rho_conjugate}
The subgroups $ N_{s} $ are $ \rho $-conjugate and mutually $ G $-isomorphic. 
\end{proposition}
\begin{proof}
Let $ C(\sigma)$ be the inner automorphism of $ G $ arising from $ \sigma $. Then $ C(\sigma)(\sigma)=\sigma $ and $ C(\sigma)(\tau) = \sigma^{1-g}\tau $. It follows that for each $ m \in \mathbb{N} $ we have $ C(\sigma)^{m}(\tau) = \sigma^{(1-g)m}\tau $. Since $ 1-g $ is coprime to $ p $, there exists $ m \in \mathbb{N} $ such that $ C(\sigma)^{m}(\tau) = \sigma\tau $, and hence $ \psi $ is an inner automorphism of $ G $. Now Proposition \ref{prop_inner_rho_conjugate} implies that the subgroups $ N_{s} $ are $ \rho $-conjugate and mutually $ G $-isomorphic. 
\end{proof}

\noindent  \textbf{The case} $ \mathbf{G \cong N \cong M: } $
There are two distinguished braces with these properties and two further families, each of size $ q-2 $. The distinguished braces are the trivial brace for $ M $ and the almost trivial brace for $ M $ (which is isomorphic to the opposite of the trivial brace). Assuming $ q > 2 $, the first family consists of the braces $ A_{t} $ (for $ 2 \leq t \leq q-1 $) constructed in  \cite[part (iii) of the second bullet point of the main theorem]{AB20}. (We have already seen that the brace $ A_{q} $ of loc. cit. has cyclic circle group.) We present these as $ \B_{t} = (B,\cdot,\circ) $ where 
\[ \begin{array}{lll}
\sigma^{i}\tau^{j} \cdot \sigma^{k}\tau^{\ell} & = & \sigma^{i+kg^{jt}} \tau^{j+\ell} \\
\sigma^{i}\tau^{j} \circ \sigma^{k}\tau^{\ell} & = & \sigma^{i+kg^{j}} \tau^{j+\ell}
\end{array} \]
To reconcile this presentation with that give in loc. cit. we apply the map $ \sigma^{i}\tau^{j} \mapsto \sigma^{i}\tau^{jt^{-1}} $.

The second family consists of the braces $ A_{t} $ (for $ 2 \leq t \leq q-1 $) constructed in  \cite[part (iv) of the second bullet point of the main theorem]{AB20}. (We have already seen that the brace $ A_{q} $ of loc. cit. is isomorphic to the almost trivial brace for $ M $.)  We obtain these braces by taking the opposites to the braces $ \B_{t} $ above: we have $ \B_{t}^{opp} = (B,\cdot',\circ) $, where
\[ \begin{array}{lll}
\sigma^{i}\tau^{j} \cdot' \sigma^{k}\tau^{\ell} & = & \sigma^{k+ig^{\ell t}} \tau^{j+\ell} \\
\sigma^{i}\tau^{j} \circ \sigma^{k}\tau^{\ell} & = & \sigma^{i+kg^{j}} \tau^{j+\ell}
\end{array} \]
To reconcile this presentation with that give in loc. cit. we apply the map  $ \sigma^{i}\tau^{j} \mapsto \sigma^{i}\tau^{jt} $.

We determine all of the regular $ G $-stable regular subgroups of $ \Perm(G) $ that yield each of these braces. We know that $\lambda(G)$ yields the trivial brace on $ G $ and that $\rho(G)$ yields the almost trivial brace on $ G $, and that each of these in a brace class by itself.

Assuming $ q > 2 $, the regular $ G $-stable subgroup of $ \Perm(G) $ obtained from the dot operation in $ \B_{t} $ is $ N_{t} = \langle \eta, \pi_{t} \rangle $, where
\[ \begin{array}{l}
\eta(\sigma^{k}\tau^{\ell}) = \sigma \cdot \sigma^{k}\tau^{\ell} = \sigma^{k+1}\tau^{\ell} \\
\pi_{t}(\sigma^{k}\tau^{\ell}) = \tau \cdot \sigma^{k}\tau^{\ell} = \sigma^{kg^{t}}\tau^{\ell+1}.
\end{array} \]
Now considering automorphisms of $ G $ we see that in all cases the automorphism $ \phi $ respects $ \cdot $, but the automorphism $ \psi $ does not; we write $ N_{t,u} = \psi^{-u} N_{t} \psi^{u} $ for $ 0 \leq u \leq p-1 $. Allowing $ t,u $ to vary we obtain a family of $ p(q-2) $ regular $G$-stable subgroups of $\Perm(G)$ that are isomorphic to $ N $. We note that these subgroups have $ p $ $ \lambda $-points, namely the elements $ \eta^{i} $; it follows that these subgroups coincide with the family described in \cite[Lemma 5.2, Equation (5.7)]{By04}. 

The $ G $-stable regular subgroup of $ \Perm(G) $ obtained from the dot operation in $ \B^{opp}_{t} $ is $ N^{opp}_{t} = \langle \eta_{t}, \pi \rangle $, where
\[ \begin{array}{l}
\eta_{t}(\sigma^{k}\tau^{\ell}) = \sigma \cdot' \sigma^{k}\tau^{\ell} = \sigma^{k+g^{\ell t}}\tau^{\ell} \\
\pi(\sigma^{k}\tau^{\ell}) = \tau \cdot' \sigma^{k}\tau^{\ell} = \sigma^{k}\tau^{\ell+1}.
\end{array} \]
As above we find that $ \phi $ respects $ \cdot' $, but $ \psi $ does not; we write $ N^{opp}_{t,u} = \psi^{-u} N^{opp}_{t} \psi^{u} $ for $ 0 \leq u \leq p-1 $, and obtain a family of $ p(q-2) $ regular $G$-stable subgroups of $\Perm(G)$ that are isomorphic to $ N $. We note that these subgroups have $ q $ $ \rho $-points, namely the elements $ \pi^{j} $; it follows that these subgroups coincide with the family described in \cite[Lemma 5.4, Equation (5.12)]{By04}. Together with the subgroups $ \rho(G) $ and $ \lambda(G) $, the subgroups of the form $ N_{t,u} $ and $ N^{opp}_{t,u} $ account for all of the $ G $-stable regular subgroups of $ \Perm(G) $ that are isomorphic to $ N $. 

\begin{proposition}
The subgroups $ N_{t,u} $ are mutually $ G $-isomorphic. The $ G $-isomorphism classes amongst the subgroups $ N^{opp}_{t,u} $ are determined by $ t $. No subgroup of the form $ N_{t,u} $ is $ G $-isomorphic to a subgroup of the form $ N^{opp}_{t,u} $.
\end{proposition}
\begin{proof}
To establish the first claim we show that the subgroups $ N_{t,u} $ arise via abelian fixed point free endomorphisms of $ G $ (see Section \ref{sec_fpf}). For fixed $ t $ consider the subgroup $ N_{t} = \langle \eta, \pi_{t} \rangle $ obtained from the dot operation in $ \B_{t} $. It is clear that $ \eta = \lambda(\sigma) $, where $ \lambda $ denotes the left regular representation of $ G $. Similarly, letting $ r $ denote the inverse of $ t $ modulo $ q $, we have
\begin{eqnarray*}
\pi_{t}^{r}(\sigma^{k}\tau^{\ell}) & = & \sigma^{kg^{rt}}\tau^{\ell+r} \\
& = &  \sigma^{kg}\tau^{\ell+r} \\
& = & \tau \sigma^{k}\tau^{\ell} \tau^{r-1} \\
& = & \lambda(\tau) \rho(\tau^{1-r}) (\sigma^{k}\tau^{\ell}),
\end{eqnarray*}
where $ \rho $ denotes the right regular representation of $ G $. Hence $ \pi_{t}^{r} = \lambda_{\circ}(\tau) \rho_{\circ}(\tau^{1-r}) $. Now consider the function  $ \psi: G \rightarrow G $ defined by $ \psi(\sigma^{k}\tau^{\ell}) = \tau^{\ell(1-r)} $. This is an abelian abelian fixed-point-free endomorphism (see \cite{KST}), and the corresponding regular $ G $-stable subgroup $ N_{\psi} $ of $ \Perm(G) $ is generated by $ \lambda_{\circ}(\sigma) $ and $ \lambda(\tau)\rho(\psi(\tau)) = \pi_{t}^{r} $. Thus $ N_{\psi} = N_{t} $, and applying Corollary \ref{cor_fpf_G_iso} we see that $ N_{t} $ is $ G $-isomorphic to $ \lambda(G) $. Since this holds for all $ t $, the subgroups $ N_{t} $ are mutually $ G $-isomorphic. Finally, applying Proposition \ref{prop_fpf_subgroups} we see that all of the subgroups $ N_{t,u} $ arise via abelian fixed point free endomorphisms, and so are mutually $ G $-isomorphic. 

Next we consider the subgroups $ N^{opp}_{t,u} $. For each fixed $ t $ we have $ N^{opp}_{t,u} = \psi^{-u} N^{opp}_{t} \psi^{u} $ for $ 0 \leq u \leq p-1 $, and we saw in the proof of Proposition \ref{prop_N_c_rho_conjugate} that $ \psi $ is an inner automorphism of $ G $; hence by Proposition \ref{prop_inner_rho_conjugate} the subgroups $ N^{opp}_{t,u} $ are all $ \rho $-conjugate and $ G $-isomorphic. To show that there are no further $ G $-isomorphisms within this family, it is sufficient to show that no two of the subgroups $ N^{opp}_{t} $ are $ G $-isomorphic. We calculate the action of $ \tau $ on $ \eta_{t} $:
\begin{eqnarray*}
\,^{\tau}\eta_{t}(\sigma^{k}\tau^{\ell}) & = & \tau \eta_{t}( \tau^{-1} \sigma^{k} \tau^{\ell} ) \\
& = & \tau \eta_{t}( \sigma^{kg^{-1}} \tau^{\ell-1} ) \\
& = & \tau ( \sigma^{kg^{-1}g^{(\ell-1)t}} \tau^{\ell-1} ) \\
& = & \sigma^{kg^{(\ell-1)t}} \tau^{\ell} ) \\
& = & \eta_{t}^{g^{-t}}(\sigma^{k}\tau^{\ell}). 
\end{eqnarray*}
Hence $ \,^{\tau}\eta_{t} = \eta_{t}^{g^{-t}} $. Now if $ \theta: N^{opp}_{t_{1}} \rightarrow N^{opp}_{t_{2}} $ is a $ G $-isomorphism then $ \theta(\eta_{t_{1}}) = \eta_{t_{2}}^{v} $ for some $ v = 1, \ldots ,p-1 $, since $ \eta_{t_{1}} $ has order $ p $. Taking the $ G $ action into account, we have
\[ \begin{array}{rll}
& \theta( \,^{\tau} \eta_{t_{1}} ) = \theta( \eta_{t_{1}}^{g^{-t_{1}}} ) = \eta_{t_{2}}^{vg^{-t_{1}}} \\
\mbox{ and } & \,^{\tau} \theta( \eta_{t_{1}} ) = \,^{\tau} \eta_{t_{2}}^{v} = \eta_{t_{2}}^{vg^{-t_{2}}}.
\end{array} \]
Since $ v \not \equiv 0 \pmod{p} $, this implies that $ g^{-t_{1}} \equiv g^{-t_{2}} \pmod{p} $, and since $ g $ has order $ q $ modulo $ p $ this implies that $ t_{1}=t_{2} $. Therefore the $ G $-isomorphisms amongst the subgroups $ N^{opp}_{t,u} $ are as described in the statement of the proposition. 

The final claim follows from Proposition \ref{prop_rho_points_isomorphism} since each of the subgroups $ N^{opp}_{s,t} $ has $ q $ $ \rho $-points, whereas each of the subgroups $ N_{s,t} $ has none. 
\end{proof}

\section*{Acknowledgements}

Funding: The second author was supported in part by the London Mathematical Society [Grant No. \# 41847].
The first author would like to thank Keele University for its hospitality during the development of this paper.

\end{document}